\documentclass[a4paper,12pt]{amsart}
\usepackage[utf8]{inputenc}
\usepackage[margin=2.5cm]{geometry}
\usepackage{amsmath}
\usepackage{amsfonts}
\usepackage{amssymb}
\usepackage{amsthm}
\usepackage{tikz-cd} % package for category theory diagrams
\usepackage{mathtools}
\usepackage{todonotes}
\usepackage{hyperref} % clickable references and toc (NOTE: SHOULD LOAD THIS LAST)

\theoremstyle{plain}
\newtheorem*{thm*}{Theorem}
\newtheorem{thm}{Theorem}[section]
\newtheorem{prop}[thm]{Proposition}
\newtheorem{lem}[thm]{Lemma}
\newtheorem{cor}[thm]{Corollary}

\theoremstyle{definition}

\newtheorem{dfn}[thm]{Definition}

\newtheorem{rem}[thm]{Remark}

\newcommand{\beq}{\begin{equation}}
\newcommand{\eeq}{\end{equation}}

\newcommand{\refnum}[2]{\ref{#1}\eqref{#2}}

\newcommand{\ZZ}{\mathbb{Z}}

\newcommand{\NN}{\mathbb{N}}
\newcommand{\QQ}{\mathbb{Q}}

\newcommand{\kk}{\Bbbk}

\newcommand{\g}{\mathfrak{g}}
\newcommand{\h}{\mathfrak{h}}
\DeclareMathOperator{\ad}{ad}

\newcommand{\mf}{\mathfrak}

\newcommand{\nonzero}{\setminus \{0\}}

\DeclareMathOperator{\im}{Im}
\DeclareMathOperator{\Ann}{Ann}

\DeclareMathOperator{\Kdim}{Kdim}

\DeclareMathOperator{\spn}{span}

\DeclareMathOperator{\Der}{Der}

\DeclareMathOperator{\Inn}{Inn}

\DeclareMathOperator{\ab}{ab}

\newcommand{\inj}{\hookrightarrow}
\newcommand{\surj}{\twoheadrightarrow}

\newcommand{\ideal}{\unlhd}

\newcommand\restr[2]{{
  \left.\kern-\nulldelimiterspace % automatically resize the bar with \right
  #1 % the function
  \vphantom{\big|}
  \right|_{#2}
  }}

\title[Enveloping algebras of derivations]{Enveloping algebras of derivations of commutative and noncommutative algebras}
\author{Jason Bell}
\author{Lucas Buzaglo}
\date{}


\begin{document}

\begin{abstract}
    Let $\kk$ be a field of characteristic zero. Motivated by the fundamental question of whether it is possible for the universal enveloping algebra of an infinite-dimensional Lie algebra to be noetherian, we study Lie algebras of derivations of associative algebras. The main result of this paper is that the universal enveloping algebra of the Lie algebra of derivations of a finitely generated $\kk$-algebra is not noetherian. This extends a result of Sierra and Walton on the Witt algebra, as well as a result of the second author on Krichever--Novikov algebras. We highlight that the result applies to derivations of both commutative and noncommutative algebras without restriction on their growth.
\end{abstract}

\maketitle

\section{Introduction}

Throughout, we let $\kk$ be a field of characteristic zero. For brevity, we say that a ring is \emph{noetherian} if it is both left and right noetherian.

Although enveloping algebras of finite-dimensional Lie algebras are fundamental examples of well-behaved noncommutative rings, those of infinite-dimensional Lie algebras remain mysterious. It has only been within the last ten years that we have begun to understand their ring theoretic properties \cite{SierraWalton, PenkovPetukhov, BiswalSierra, BellBuzaglo, AndruskiewitschMathieu}. One of the most important open problems regarding enveloping algebras of infinite-dimensional Lie algebras is the question of noetherianity: it is not known whether it is possible for such an algebra to be noetherian. This is a famously difficult question which has been posed by many authors \cite{GoodearlWarfield, Brown, Goodearl, BrownGilmartin, Andruskiewitsch}, but its earliest appearance was fifty years ago in Amayo and Stewart's book on infinite-dimensional Lie algebras \cite[Question 27]{AmayoStewart}.

Many of the most significant examples of infinite-dimensional Lie algebras arise as derivations of associative algebras. For example, the Virasoro algebra, a hugely important infinite-dimensional Lie algebra in both mathematics and physics, is the universal central extension of the Witt algebra $W = \Der(\kk[t,t^{-1}])$. Sierra and Walton proved that the enveloping algebras of the Witt and Virasoro algebras are not noetherian \cite[Theorem 0.5, Corollary 0.6]{SierraWalton}. The second author then extended the result, proving that the enveloping algebra of derivations of a finitely generated commutative domain of Krull dimension 1 is not noetherian \cite{Buzaglo}.

The present paper is devoted to further extending this result to include derivations of algebras which are not necessarily commutative or domains, with no restriction on their (Krull or Gelfand--Kirillov) dimension.

\begin{thm}[Theorem \ref{thm:main}]\label{thm:intro main}
    Let $A$ be an infinite-dimensional finitely generated $\kk$-algebra. Then $U(\Der(A))$ is not noetherian. Furthermore, $U(\Inn(A))$ is not noetherian whenever $\Inn(A)$ is infinite-dimensional.
\end{thm}

We refer the reader to Definition \ref{def:derivation} for the definitions of $\Der(A)$ and $\Inn(A)$.

Theorem \ref{thm:intro main} provides one of the first general non-noetherianity results for enveloping algebras of Lie algebras that do not rely on either an explicit grading on the Lie algebra or its simplicity. Previous work by Sierra and Walton \cite{SierraWalton} on the enveloping algebra of the Witt algebra, as well as the later result by Andruskiewitsch and Mathieu \cite[Theorem 1.3]{AndruskiewitschMathieu}, which demonstrates that enveloping algebras of $\ZZ^n$-graded simple Lie algebras are non-noetherian, both heavily depend on the graded structure of the Lie algebra.

In contrast, there are very few examples of infinite-dimensional ungraded or non-simple Lie algebras whose enveloping algebras are known to be non-noetherian. The Lie algebras of derivations presented in Theorem \ref{thm:intro main} significantly expand this class, offering a wealth of new examples in this direction.

We remark that Theorem \ref{thm:intro main} is not true without the finitely generated and infinite-dimensional assumptions on $A$. For example, if $A = \overline{\QQ}$ and $\kk = \QQ$ then $A$ has no non-trivial $\kk$-linear derivations. Even when $\kk$ is algebraically closed there are counterexamples: see \cite[Corollary 4.5]{MukhopadhyaySmith} for a noetherian example over an arbitrary field.

The proof of Theorem \ref{thm:intro main} is achieved as follows: in the commutative case, we are able to reduce to the case where $A$ is a domain by considering the quotient $A/P$ for some minimal prime $P$ of $A$. Since derivations preserve minimal primes, we get a map
$$\varphi \colon \Der(A) \to \Der(A/P)$$
whose image can be shown to have a non-noetherian enveloping algebra. In the case where $A/P$ has Krull dimension 1, this follows by results from \cite{Buzaglo}. If $A/P$ has Krull dimension 2 or more, the proof is significantly easier. It then follows that $U(\Der(A))$ is not noetherian.

The proof of Theorem \ref{thm:intro main} in the noncommutative case presents more of a challenge, so it is split into several sub-cases.
\begin{enumerate}
    \item $A$ is a finitely generated module over $Z(A)$:
    \begin{enumerate}
        \item $\dim_\kk(A/Z(A)) < \infty$ (Proposition \ref{prop:finite over center}).\label{case:finite codimension center}
        \item $\dim_\kk(A/Z(A)) = \infty$ (Lemma \ref{lem:finite over center}).\label{case:finite over center}
    \end{enumerate}
    \item $A$ is not a finitely generated module over $Z(A)$:
    \begin{enumerate}
        \item $\dim_\kk(Z(A)) < \infty$ (Proposition \ref{prop:finite-dimensional center}).\label{case:finite-dimensional center}
        \item $\dim_\kk(Z(A)) = \infty$ (Proposition \ref{prop:not finite over center}).\label{case:not finite over center}
    \end{enumerate}
\end{enumerate}

In case \eqref{case:finite codimension center}, we are able to emulate the proof outlined above with the homomorphism $\Der(A) \to \Der(A/P)$, for some minimal prime $P$ of $A$. However, proving that the image of this homomorphism has a non-noetherian enveloping algebra requires a much more technical argument.

The remaining cases \eqref{case:finite over center}--\eqref{case:not finite over center} all have infinite-dimensional spaces of inner derivations, which we denote by $\Inn(A)$. The Lie algebraic properties of $\Inn(A)$ are closely related to ring theoretic properties of $A$, since $\Inn(A) \cong A/Z(A)$ (where we view $A$ as a Lie algebra with commutator bracket). We are thus able to prove that $U(\Inn(A))$ is not noetherian by exploiting classical results from noncommutative ring theory.

We finish the paper with the following consequence of Theorem \ref{thm:intro main}.

\begin{cor}[Corollary \ref{cor:associative}]
    Let $A$ be an infinite-dimensional $\kk$-algebra, and view $A$ as a Lie algebra with commutator bracket. Then $U(A)$ is not noetherian.
\end{cor}

\vspace{2mm}

\noindent \textbf{Acknowledgments:} Some of this research was done as part of the second author's PhD at the University of Edinburgh, under the supervision of Prof. Susan J. Sierra. Part of this work was carried out during the second author's visit to the University of Waterloo; he would like to thank the institution for their hospitality and support. This collaboration was made possible by the generous funding provided through the University of Edinburgh's Laura Wisewell Travel Scholarship, and NSERC grant RGPIN-2022-02951.

\section{Preliminaries}

For brevity, we say \emph{affine algebra} to refer to a finitely generated $\kk$-algebra. When we say that a (Lie or associative) algebra is (in)finite-dimensional, we always mean (in)finite-dimensional as a $\kk$-vector space. Thus, when referring to other notions of dimension, such as Krull dimension, we mention this explicitly.

In this section, we define the objects of interest for this paper, state the main result of this paper, and mention some standard results which we will use throughout. We begin by recalling the definition of derivations of algebras.

\begin{dfn}\label{def:derivation}
  Let $A$ be a $\kk$-algebra. A \emph{derivation} of $A$ is a $\kk$-linear map $d \colon A \to A$ such that $d(ab) = d(a)b + ad(b)$ for all $a,b \in A$. The set of all derivations of $A$ is denoted $\Der(A)$. It is a Lie algebra with commutator bracket $[d_1,d_2] = d_1 \circ d_2 - d_2 \circ d_1$.

  Letting $Z(A)$ be the center of $A$, note that $\Der(A)$ is naturally a module over $Z(A)$: for $z \in Z(A)$, $d \in \Der(A)$, we have $zd \in \Der(A)$ defined by $(zd)(a) \coloneqq z \cdot d(a)$ for $a \in A$. A Lie subalgebra of $\Der(A)$ which is also a $Z(A)$-submodule is called a \emph{submodule-subalgebra}.

  An \emph{inner derivation} of $A$ is a derivation of the form $\ad_a \colon A \to A$ defined by $\ad_a(b) = ab - ba$, where $a,b \in A$. The set of all inner derivations of $A$ is denoted $\Inn(A)$. It is a Lie ideal of $\Der(A)$.
\end{dfn}

\begin{rem}
    Let $A$ be a $\kk$-algebra. We have $\Inn(A) \cong A/Z(A)$, where $Z(A)$ is the center of $A$, and we view $A$ as a Lie algebra with commutator bracket. This isomorphism is induced by the map
    \begin{align*}
        A &\to \Inn(A) \\
        a &\mapsto \ad_a.
    \end{align*}
\end{rem}

In this paper, we are interested in Lie algebras of derivations of affine algebras, and prove the following theorem.

\begin{thm}\label{thm:main}
  Let $A$ be an infinite-dimensional affine algebra. Then $U(\Der(A))$ is not noetherian. Furthermore, $U(\Inn(A))$ is not noetherian whenever $\Inn(A)$ is infinite-dimensional (in other words, whenever $\dim_\kk(A/Z(A)) = \infty$).
\end{thm}

The following result gives a useful tool for showing that an enveloping algebra is non-noetherian. We will use it extensively to prove Theorem \ref{thm:main}.

\begin{prop}\label{prop:non-noetherian}
  Let $\g$ be a Lie algebra such that $U(\g)$ is noetherian. Then all of the following hold.
  \begin{enumerate}
    \item If $\h$ is a subalgebra of $\g$ then $U(\h)$ is noetherian.\label{item:subalgebra}
    \item If $\g$ has finite codimension in a larger Lie algebra $\h$, then $U(\h)$ is noetherian.\label{item:finite codimension}
    \item Any abelian subalgebra of $\g$ is finite-dimensional.\label{item:abelian subalgebra}
    \item The Lie algebra $\g$ has ACC on Lie subalgebras, in other words, any infinite ascending chain of subalgebras $\h_1 \subseteq \h_2 \subseteq \cdots$ of $\g$ stabilizes.\label{item:ACC on subalgebras}
    \item The abelianization $\g^{\ab} = \g/[\g,\g]$ of $\g$ is finite-dimensional.\label{item:abelianization}
    \item If $\g$ is solvable, then $\g$ is finite-dimensional.\label{item:solvable}
  \end{enumerate}
\end{prop}
\begin{proof}
  \eqref{item:subalgebra} is \cite[Lemma 1.7]{SierraWalton}, \eqref{item:finite codimension} is \cite[Proposition 2.1]{Buzaglo}, and \eqref{item:ACC on subalgebras} is \cite[Proposition 11.1.2]{AmayoStewart}.

  For \eqref{item:abelian subalgebra}, suppose $\mf{a}$ is an infinite-dimensional abelian subalgebra of $\g$. Then $U(\mf{a})$ is not noetherian, since it is isomorphic to a polynomial ring in infinitely many variables, so $U(\g)$ is not noetherian by \eqref{item:subalgebra}.

  For \eqref{item:abelianization}, suppose the abelianization $\g^{\ab} = \g/[\g,\g]$ is infinite-dimensional. Now, $U(\g^{\ab})$ is not noetherian by \eqref{item:abelian subalgebra}. The natural map $U(\g) \to U(\g^{\ab})$ is surjective, so $U(\g)$ is also non-noetherian.

  \eqref{item:solvable} is \cite[Corollary 2.4]{AmayoStewart2}, but we give a short proof here regardless. Let $D^0(\g) = \g$ and for $n \in \NN$, let $D^{n+1}(\g) = [D^n(\g),D^n(\g)]$. Suppose $\g$ is infinite-dimensional and solvable, so there exists $N \in \NN$ such that $D^N(\g) = 0$. Let $n \in \NN$ be minimal such that $D^{n+1}(\g)$ is finite-dimensional, and let $\h = D^n(\g)$. Then $\h$ is an infinite-dimensional Lie algebra whose derived subalgebra is finite-dimensional, so $U(\h)$ is not noetherian by \eqref{item:abelianization}. Therefore, $U(\g)$ is not noetherian by \eqref{item:subalgebra}.
\end{proof}

We prove Theorem \ref{thm:main} through a series of lemmas and propositions depending on the form that the algebra $A$ takes.

\section{Proof of Theorem \ref{thm:main} in the commutative case}

The first case we consider for the proof of Theorem \ref{thm:main} is where the $\kk$-algebra $A$ is a commutative domain. The most difficult situation is when $\Kdim(A) = 1$, which was proved in \cite{SierraWalton} for $A = \kk[t]$ and $\kk[t,t^{-1}]$, and later completed in \cite{Buzaglo} for an arbitrary affine commutative domain of Krull dimension 1. If $A$ has Krull dimension at least 2, the proof follows easily from Proposition \refnum{prop:non-noetherian}{item:abelianization} by constructing a subalgebra of $\Der(A)$ with infinite-dimensional abelianization.

\begin{lem}\label{lem:commutative domain}
  Let $A$ be an affine commutative domain such that $\dim_\kk(A) = \infty$, and let $\g$ be a nonzero submodule-subalgebra of $\Der(A)$. Then $U(\g)$ is not noetherian. In particular, $U(\Der(A))$ is not noetherian.
\end{lem}
\begin{proof}
  If $\Kdim(A) = 1$, then $U(\Der(A))$ is not noetherian by \cite[Theorem 3.3]{Buzaglo}. In this case, $\g$ must have finite codimension in $\Der(A)$, so $U(\g)$ is not noetherian by Proposition \refnum{prop:non-noetherian}{item:finite codimension}. Therefore, assume $A$ has Krull dimension at least 2.

  Let $f \in A \nonzero$ be a non-unit. Then $(f)$, the ideal of $A$ generated by $f$, has infinite codimension in $A$. Let $d \in f \cdot \g \nonzero$, and let $\h = A \cdot d$. Since $\g$ is an $A$-module, it follows that $\h$ is a Lie subalgebra of $\g$ with Lie bracket
  $$[gd,hd] = (gd(h) - d(g)h)d$$
  for $g,h \in A$. Since $d \in f \cdot \g$, we see that $gd(h) - d(g)h \in (f)$. Hence, the derived subalgebra $[\h,\h]$ is contained in $(f)d$ and therefore has infinite codimension in $\h$. By Proposition \refnum{prop:non-noetherian}{item:abelianization}, $U(\h)$ is not noetherian. It follows that $U(\g)$ is not noetherian by Proposition \refnum{prop:non-noetherian}{item:subalgebra}.
\end{proof}

\begin{rem}\label{rem:submodule-subalgebra}
    Let $B$ be an affine $\kk$-subalgebra of an affine commutative domain $A$ such that $\dim_\kk(A/B) < \infty$. The same proof as above works to show that $U(\g)$ is not noetherian if $\g$ is a nonzero Lie subalgebra of $\Der(A)$ which is also a $B$-submodule.
\end{rem}

We move on to the case where $A$ is commutative but not necessarily a domain. To prove this, we will need the following result, which states that minimal primes of a $\kk$-algebra $A$ are preserved by derivations of $A$.

\begin{prop}[{\cite[Proposition 14.2.3]{McConnellRobson}}]\label{prop:minimal primes}
    Let $A$ be a $\kk$-algebra and let $P$ be a minimal prime of $A$. Then $d(P) \subseteq P$ for all $d \in \Der(A)$.
\end{prop}

We are now ready to complete the proof of Theorem \ref{thm:main} in the case where $A$ is commutative. This is achieved by applying Proposition \ref{prop:minimal primes} to construct a Lie algebra homomorphism $\Der(A) \to \Der(A/P)$ for some minimal prime $P$, and then using Lemma \ref{lem:commutative domain} to conclude that $U(\Der(A))$ is not noetherian. We can do this because, as we will see, the image of the homomorphism is a submodule-subalgebra of $\Der(A/P)$. There is one subtlety: the Lie algebra homomorphism might be zero, so we need to be careful.

\begin{prop}\label{prop:commutative}
    Let $A$ be an affine infinite-dimensional commutative algebra. Then $U(\Der(A))$ is not noetherian.
\end{prop}
\begin{proof}
    Let $P$ be a minimal prime of $A$ such that $\dim_\kk(A/P) = \infty$. Such a $P$ must exist, since $A$ is infinite-dimensional. Indeed, if $A/Q$ is finite-dimensional for all minimal primes $Q$ of $A$, then $\Kdim(A) = 0$, a contradiction.

    Now, $P$ is a minimal prime and $A$ is noetherian, so $P$ is associated \cite[Theorem 3.1]{Eisenbud}. In other words, there exists $x \in A$ such that $P = \Ann(x)$. Given $d \in \Der(A/P)$, we define a derivation $d' \in \Der(A)$ as follows: if $d(a + P) = b + P$, where $a,b \in A$, then we set $d'(a) = bx$. It is easy to see that $d'$ is a well-defined derivation of $A$. Let
    \begin{align*}
        \tau \colon \Der(A/P) &\to \Der(A) \\
        d &\mapsto d'.
    \end{align*}
    We emphasize that $\tau$ is only a $\kk$-linear map, not a homomorphism of Lie algebras.
    
    We claim that $\tau$ is injective. Let $d \in \ker(\tau)$ and let $a,b \in A$ such that $d(a + P) = b + P$. Since $d' = 0$, we have $0 = d'(a) = bx$, so $b \in \Ann(x) = P$. But then $d(a + P) = b + P = 0$. Now, this holds for any $a \in A$, so we conclude that $d = 0$. Hence, $\ker(\tau) = 0$, as claimed.
    
    Suppose $x \in P$, in other words, $x^2 = 0$. Note that, by definition of $d'$, we have $d'(x) = 0$ for all $d \in \Der(A/P)$. Let $d_1,d_2 \in \Der(A/P)$ and $a \in A$. Let $b,c \in A$ such that $d_1(a + P) = b + P$ and $d_2(a + P) = c + P$. Then
    \begin{align*}
        [d_1',d_2'](a) &= d_1'(d_2'(a)) - d_2'(d_1'(a)) = d_1'(cx) - d_2'(bx) \\
        &= d_1'(c)x + cd_1'(x) - d_2'(b)x - bd_2'(x) = 0,
    \end{align*}
    since $d_i'(x) = 0$ and $d_1'(c),d_2'(b) \in Ax$. Therefore, $[d_1',d_2'] = 0$ for all $d_1,d_2 \in \Der(A/P)$. It follows that $\mf{a} = \im(\tau) \subseteq \Der(A)$ is an infinite-dimensional abelian Lie algebra, since $\tau$ is injective. Hence, $U(\Der(A))$ is not noetherian, by Proposition \ref{prop:non-noetherian}.
    
    Therefore, we may assume that $x \notin P$. By Proposition \ref{prop:minimal primes}, $P$ is preserved by all derivations of $A$. Therefore, any $d \in \Der(A)$ induces a derivation $\overline{d}$ of $A/P$ as follows: for $a \in A$, we let
    $$\overline{d}(a + P) = d(a) + P.$$
    This gives a Lie algebra homomorphism
    \begin{align*}
    \varphi \colon \Der(A) &\to \Der(A/P) \\
    d &\mapsto \overline{d}.
    \end{align*}
    Let $\g = \im(\varphi) \subseteq \Der(A/P)$. We claim that $\g$ is a submodule-subalgebra of $\Der(A/P)$. Indeed, if $d \in \Der(A)$ and $a \in A$, then $(a + P)\overline{d} = \overline{ad} \in \g$.
    
    Note that $\g \neq 0$: if $d \in \Der(A/P) \nonzero$, then $\varphi(\tau(d)) = (x + P)d$ is a nonzero element of $\g$, since $x \notin P$ and $A/P$ is a domain. Now, $\g$ is a nonzero submodule-subalgebra of $\Der(A/P)$, so we can apply Lemma \ref{lem:commutative domain} to deduce that $U(\g)$ is not noetherian. Since $U(\Der(A))$ surjects onto $U(\g)$, it follows that $U(\Der(A))$ is not noetherian.
\end{proof}

\section{Proof of Theorem \ref{thm:main} in the noncommutative case}

We move on to the case where $A$ is not necessarily commutative. Proposition \ref{prop:minimal primes} still applies to noncommutative algebras, so we still get a Lie algebra homomorphism
$$\varphi \colon \Der(A) \to \Der(A/P)$$
if $P$ is a minimal prime of the noncommutative algebra $A$. Therefore, we could try to emulate the proof of Proposition \ref{prop:commutative}, but we run into some issues. First, to apply Lemma \ref{lem:commutative domain}, we need $\g$ to be a submodule-subalgebra of $A/P$. Unfortunately, this is not quite true: if $a \in A \setminus Z(A)$ and $d \in \Der(A)$, then $ad$ is not a derivation in general, so we cannot say that $(a + P)\varphi(d) = \varphi(ad)$ like we did in the commutative case. However, it is still true that $\g$ is a module over $\varphi(Z(A))$. Therefore, we can apply Remark \ref{rem:submodule-subalgebra} instead, provided $\varphi(Z(A))$ has finite codimension in $A/P$. This means that we must assume that $Z(A)$ has finite codimension in $A$ for this proof.

The second issue is that, to conclude that $U(\g)$ is not noetherian using Remark \ref{rem:submodule-subalgebra}, we need $A/P$ to be commutative. This is easily resolved.

\begin{lem}\label{lem:prime quotient}
    Let $A$ be an affine algebra with $\dim_\kk(A/Z(A)) < \infty$, and let $P$ be a prime ideal of $A$ of infinite codimension. Then $A/P$ is commutative.
\end{lem}
\begin{proof}
    Let $B = A/P$. Assume, for a contradiction, that $B$ is not commutative, so there exists $x \in B \setminus Z(B)$. Note that $\Inn(B) \cong B/Z(B)$ is finite-dimensional, since $Z(A)$ has finite codimension in $A$. It follows that there must be an element $z \in Z(B) \nonzero$ such that $xz \in Z(B)$. This is because $z$ is simply an element in the kernel of the map
    \begin{align*}
        Z(B) &\to \Inn(B) \\
        w &\mapsto \ad_{xw}.
    \end{align*}
    This kernel must be nonzero by comparing dimensions. Now, $x$ is non-central in $B$, so there exists $y \in B$ such that $[x,y] \neq 0$. Since $xz$ is central, it follows that $[xz,y] = 0$, which gives $z[x,y] = 0$. Hence, $z$ is a zero-divisor. But $B$ is a prime ring, so $Z$ consists of regular elements in $B$. This is a contradiction, so the result follows.
\end{proof}

There is still the issue of $A/P$ being infinite-dimensional. To prove that a minimal prime ideal of infinite codimension exists under our assumptions, we will use the theory of rings satisfying a polynomial identity (\emph{PI rings} for short) -- see \cite{Rowen} or \cite[Chapter 13]{McConnellRobson} for an introduction to the subject. We will use the following standard facts about PI rings implicitly in the remainder of this paper.

\begin{lem}[{\cite[Lemma 13.1.7, Corollary 13.1.13]{McConnellRobson}}]
    Let $A$ and $B$ be $\kk$-algebras.
    \begin{enumerate}
        \item If $A$ is PI and $B \subseteq A$, then $B$ is also PI.
        \item If $A$ and $B$ are PI, then $A \times B$ is PI.
        \item If $A$ is a finitely generated module over $Z(A)$, then $A$ is PI.
    \end{enumerate}
\end{lem}

Since we are assuming that $A/Z(A)$ is finite-dimensional, certainly $A$ is a finitely generated module over $Z(A)$, so $A$ is PI. Using this, it will follow that such a minimal prime ideal $P$ exists.

The final hurdle to overcome is that $\g$ must be nonzero. This presents another difficulty, because we cannot use the fact that minimal primes are associated, as we did in the commutative case. As we prove next, it is possible to overcome this issue.

\begin{prop}\label{prop:finite over center}
    Let $A$ be an infinite-dimensional affine algebra such that $Z(A)$ has finite codimension in $A$. Then $U(\Der(A))$ is not noetherian.
\end{prop}
\begin{proof}
    We have $\kk \subseteq Z(A) \subseteq A$, where $A$ is finitely generated as a $\kk$-algebra and as a $Z(A)$-module. By the Artin--Tate lemma \cite[Lemma 13.9.10]{McConnellRobson}, $Z(A)$ is finitely generated as a $\kk$-algebra and $A$ is noetherian. Furthermore, $A$ is PI since it is a finitely generated $Z(A)$-module.

    Assume, for a contradiction, that $U(\Der(A))$ is noetherian. By Proposition \ref{prop:commutative}, it follows that $Z(A) \neq A$. Since $Z(A)$ is affine and $\dim_\kk(Z(A)) = \infty$, there exists a minimal prime $\mathfrak{p}$ of $Z(A)$ such that $\dim_\kk(Z(A)/\mathfrak{p}) = \infty$. Now, $A$ is PI, so we can apply \cite[Theorem 13.8.14]{McConnellRobson} to deduce that there is a minimal prime $P$ of $A$ such that $P \cap Z(A) = \mathfrak{p}$. Note that
    $$\dim_\kk(A/P) \geq \dim_\kk((Z(A) + P)/P) = \dim_\kk(Z(A)/\mf{p}) = \infty.$$
    By Lemma \ref{lem:prime quotient}, the quotient $A/P$ is commutative. Let $Z$ be the image of $Z(A)$ in $A/P$. Since $\dim_\kk(A/Z(A)) < \infty$, it follows that $Z$ has finite codimension in $A/P$.
    
    By Proposition \ref{prop:minimal primes}, derivations of $A$ preserve $P$. Therefore, we get a map
    $$\varphi \colon \Der(A) \to \Der(A/P)$$
    defined as was done in the proof of Proposition \ref{prop:commutative}. Let $\g = \im(\varphi)$. Since $\Der(A)$ surjects onto $\g$, it follows that $U(\g)$ is noetherian.
    
    Similarly to the proof of Proposition \ref{prop:commutative}, we can easily see that $\g$ is a $Z$-submodule of $\Der(A/P)$. Since $Z$ has finite codimension in $A/P$, it follows by the observation in Remark \ref{rem:submodule-subalgebra} that $\g$ must be zero. This implies that $d(A) \subseteq P$ for all $d \in \Der(A)$.

    Define
    $$X = \{J \ideal A \mid J \subseteq P \text{ and the map } \Der(A/J) \to \Der(A/P) \text{ has zero image}\}.$$
    By the above, $0 \in X$, so $X \neq \emptyset$. Since $A$ is noetherian, $X$ must have a maximal element $J$. Let $B = A/J$, and let $Q = P/J$ be the image of $P$ in $B$. Then $B$ has the property that the map
    $$\psi \colon \Der(B) \to \Der(B/Q) \cong \Der(A/P)$$
    is zero, but if $I$ is a nonzero ideal of $B$ contained in $Q$, then $\Der(B/I) \to \Der(B/Q)$ has nonzero image.

    Let $x \in B \setminus Z(B)$ and let $V = \{z \in Z(B) \mid xz \in Z(B)\}$. Then $V$ is a subspace of $Z(B)$ of finite codimension, since it is the kernel of the linear map
    \begin{align*}
        Z(B) &\to B/Z(B) \\
        z &\mapsto xz + Z(B).
    \end{align*}
    By definition of $V$, we have $Vx \subseteq Z(B)$. Since $x$ is not central in $B$, there exists $y \in B$ such that $[x,y] \neq 0$. Then $V[x,y] = 0$, since $z[x,y] = [xz,y] = 0$ for all $z \in V$. Let $U_0 = \{z \in Z(B) \mid z[x,y] = 0\}$. Note that $V \subseteq U_0$, so $U_0$ is an ideal of $Z(B)$ of finite codimension.
    
    Since  $B/Q \cong A/P$ is commutative, $Q$ must contain all commutators in $B$, so $q \coloneqq [x,y] \in Q \nonzero$. Letting $U = BU_0$, we see that $U$ is an ideal of $B$ of finite codimension, since $U_0 \subseteq U$. By construction, $Uq = qU = 0$. Let $I = BqB \subseteq Q$ be the ideal of $B$ generated by $q$. Note that $U_0$ has finite codimension in $B$, while $Q$ does not. Therefore, there exists $u \in U_0 \subseteq Z(B)$ such that $u \notin Q$, and thus $uI = Iu = 0$.
    
    By construction, the natural map $\eta \colon \Der(B/I) \to \Der(B/Q)$ is nonzero. Let $d \in \Der(B/I)$ such that $\eta(d) \neq 0$. Thus, there exists $b \in B$ such that $d(b + I) = w + I$ for some $w \in B \setminus Q$. Define $d' \colon B \to B$ as follows: if $d(r + I) = s + I$, where $r,s \in B$, then we define $d'(r) = su$. It can easily be checked that $d'$ is a well-defined derivation of $B$.

    By construction, we have $\psi(d') = 0$, in other words, $d'(B) \subseteq Q$. In particular, $d'(b) \in Q$. By definition of $d'$, we have $d'(b) = wu$, so it follows that $wu \in Q$. Hence, $Bwu \subseteq Q$. Since $u \in Z(B)$, we therefore have $uBw \subseteq Q$. However, $u,w \notin Q$, so this contradicts the primality of $Q$.
\end{proof}

Having considered the case where $A$ is very close to being commutative, we move on to the ``next most commutative'' case: when $Z(A)$ has infinite codimension in $A$, but $A$ is still finitely generated as a module over $Z(A)$. In this case, $\Inn(A)$ is infinite-dimensional, so we prove that $U(\Inn(A))$ is not noetherian. As we will see, having access to many inner derivations will greatly simplify the proof.

\begin{lem}\label{lem:finite over center}
    Let $A$ be a finitely generated $\kk$-algebra such that $\dim_\kk(A/Z(A)) = \infty$ and $A$ is a finitely generated module over $Z(A)$. Then $U(\Inn(A))$ is not noetherian. Consequently, $U(\Der(A))$ is not noetherian.
\end{lem}
\begin{proof}
    Since $A$ is a finitely generated module over its center, we have
    $$A = Z(A)x_1 + \cdots + Z(A)x_n,$$
    for some $x_i \in A$. Take $\ad_a \in \Inn(A)$, where $a \in A$. Then
    $$a = z_1 x_1 + \cdots + z_n x_n$$
    for some $z_i \in Z(A)$. It follows that
    $$\ad_a = \ad_{z_1 x_1} + \cdots + \ad_{z_n x_n} = z_1\ad_{x_1} + \cdots + z_n\ad_{x_n},$$
    so
    $$\Inn(A) = Z(A)\ad_{x_1} + \cdots + Z(A)\ad_{x_n}.$$
    Now, $\Inn(A) \cong A/Z(A)$ is infinite-dimensional, so at least one $Z(A)\ad_{x_i}$ is infinite-dimensional. But $Z(A)\ad_{x_i}$ is an abelian Lie algebra, so $U(\Inn(A))$ is not noetherian by Proposition \refnum{prop:non-noetherian}{item:abelian subalgebra}.
\end{proof}

As in Lemma \ref{lem:finite over center}, the space of inner derivations is infinite-dimensional in all remaining cases of Theorem \ref{thm:main}. Thus, it remains to demonstrate that $U(\Inn(A))$ is not noetherian in what follows. 

Having proved Proposition \ref{prop:finite over center} and Lemma \ref{lem:finite over center}, we are left with the case where $A$ is not a finitely generated module over $Z(A)$. We split this in two sub-cases depending on whether $Z(A)$ is finite- or infinite-dimensional. We first assume that $Z(A)$ is finite-dimensional.

\begin{prop}\label{prop:finite-dimensional center}
    Let $A$ be an infinite-dimensional affine algebra with finite-dimensional center. Then $U(\Inn(A))$ is not noetherian. Consequently, $U(\Der(A))$ is not noetherian.
\end{prop}
\begin{proof}
    The map $\ad \colon A \surj \Inn(A)$ sending $x \in A$ to $\ad_x$ is a homomorphism of Lie algebras, where we view $A$ as a Lie algebra with commutator bracket. Since $\Inn(A) \cong A/Z(A)$, we may equivalently view this as the natural quotient map $A \surj A/Z(A)$. Assume, for a contradiction, that $U(\Inn(A))$ is noetherian.
    
    Suppose $A$ is not noetherian. If $I_1\subsetneqq I_2\subsetneqq \cdots $ is an infinite chain of left or right ideals of $A$, then their image under $\ad$ gives a non-terminating chain of Lie subalgebras of $\Inn(A)$, which contradicts Proposition \refnum{prop:non-noetherian}{item:ACC on subalgebras}. Therefore, $A$ is noetherian.
    
    We claim that each $a \in A$ is algebraic over $\kk$. The reason for this is that $\ad_{\kk[a]} = \spn\{\ad_{a^n} \mid n \in \NN\}$ is an abelian Lie subalgebra of $\Inn(A)$, so $\ad_{\kk[a]}$ must be finite-dimensional by Proposition \refnum{prop:non-noetherian}{item:abelian subalgebra}. Since $Z(A)$ is finite-dimensional, it follows that $\kk[a]$ is finite-dimensional, which proves the claim.
    
    Now let $N$ denote the prime radical of $A$, in other words, the intersection of the prime ideals of $A$. By Levitzki's theorem \cite[Theorem 2.3.7]{McConnellRobson}, $N$ is nilpotent. It follows that $\mf{n} \coloneqq \ad_N = \{\ad_x \mid x \in N\} \cong \frac{N + Z(A)}{Z(A)}$ is a nilpotent Lie subalgebra of $\Inn(A)$ and thus $N$ must be finite-dimensional by Proposition \refnum{prop:non-noetherian}{item:solvable}.
    
    We claim that $\dim_\kk(Z(A/N)) < \infty$. Assume, for a contradiction, that $Z(A/N)$ is infinite-dimensional. Let $q \colon A \surj A/N$ be the natural quotient map, and let
    $$W = q^{-1}(Z(A/N)) \subseteq A.$$
    Then $W$ is infinite-dimensional, and $[W,A] \subseteq N$. Now, $A$ is finitely generated, so there is a finite generating set $a_1,\ldots,a_n$ of $A$. Consider the maps $\nu_i = \restr{\ad_{a_i}}{W} \colon W \to N$. Note that $\dim_\kk(W/\ker(\nu_i)) = \dim_\kk(\im(\nu_i)) \leq \dim_\kk(N) < \infty$, so $\ker(\nu_i)$ has finite codimension in $W$ for all $i$. Therefore, the vector space
    $$W_0 = \bigcap_{i = 1}^n \ker(\nu_i) = \{w \in W \mid [w,a_i] = 0 \text{ for all } i\}$$
    has finite codimension in $W$. In particular, $W_0$ is infinite-dimensional. Now, $[W_0,A] = 0$ by definition of $W_0$, so $W_0 \subseteq Z(A)$, which contradicts the fact that $Z(A)$ is finite-dimensional. This proves the claim.
    
    Since every $x \in A/N$ is algebraic over $\kk$, we see that every regular element $x \in A/N$ is a unit. Therefore, $A/N$ is equal to its classical (left or right) quotient ring. Now, $A/N$ is noetherian, so it is a (left and right) Goldie ring. Since $A/N$ is semiprime, Goldie's theorem \cite[Theorem 6.15]{GoodearlWarfield} implies that $A/N$ is semisimple artinian. By the Artin--Wedderburn theorem, $A/N \cong \prod_{i=1}^s M_{n_i}(D_i)$ for some division rings $D_1,\ldots,D_s$ and $n_1,\ldots,n_s \ge 1$.
    
    It follows that $Z(A/N) \cong Z(D_1) \times \cdots \times Z(D_s)$, which is finite-dimensional by the above. On the other hand, $A$ is infinite-dimensional, so at least one $D_i$ is infinite-dimensional. Without loss of generality, we assume $D_1$ is infinite-dimensional. Let $K$ be a maximal subfield of $D_1$, which must be infinite-dimensional \cite[Theorem 15.8]{Lam}, and identify $\lambda \in K$ with 
    $(\lambda \cdot I_{n_1},0,0,\ldots,0) \in \prod_{i=1}^s M_{n_i}(D_i) \cong A/N$. Letting $U = q^{-1}(K)$ be the lift of $K$ in $A$, we see that $U$ is infinite-dimensional and $[U,U] \subseteq N$, since $K \subseteq A/N$ is commutative.
    
    Let $U_0 = U + N + Z(A)$, and let $\mf{u}_0 = \ad_{U_0} \cong U_0/Z(A)$. Then $[U_0,U_0] \subseteq N$, so $\mf{u}_0/\mf{n}$ is abelian and infinite-dimensional. It follows that $U(\mf{u}_0/\mf{n})$ is not noetherian by Proposition \refnum{prop:non-noetherian}{item:abelian subalgebra}. Now, $U(\mf{u}_0)$ surjects onto $U(\mf{u}_0/\mf{n})$, so $U(\mf{u}_0)$ is not noetherian. This contradicts Proposition \refnum{prop:non-noetherian}{item:subalgebra}, which concludes the proof.
\end{proof}

We are now left with the last remaining situation, where $A$ is not finite over its center but $Z(A)$ is still infinite-dimensional. Before we can prove this, we consider derivations of prime algebras.

\begin{lem}\label{lem:noncommutative prime}
    Let $A$ be a noncommutative infinite-dimensional prime affine algebra. Then $U(\Inn(A))$ is not noetherian, and thus $U(\Der(A))$ is not noetherian.
\end{lem}
\begin{proof}
    If $\dim_\kk(Z(A)) < \infty$, then we are done by Proposition \ref{prop:finite-dimensional center}. So, assume $Z(A)$ is infinite-dimensional. Since $A$ is not commutative, there exists $x \in A \setminus Z(A)$. It is easy to see that
    $$\mf{a} \coloneqq \{\ad_{xz} \mid z \in Z(A)\}$$
    is an abelian Lie subalgebra of $\Inn(A)$. Thus, if $\dim_\kk(\mf{a}) = \infty$, then we are done by Proposition \refnum{prop:non-noetherian}{item:abelian subalgebra}. Therefore, we may assume $\dim_\kk(\mf{a}) < \infty$.
    
    Since $Z(A)$ is infinite-dimensional and $\mf{a}$ is finite-dimensional, there must be an element $z \in Z(A) \nonzero$ such that $xz \in Z(A)$, because $z$ is an element in the kernel of the map
    \begin{align*}
        Z(A) &\to \mf{a} \\
        w &\mapsto \ad_{xw}.
    \end{align*}
    This kernel must be nonzero by the aforementioned dimension comparison. The proof now follows as in Lemma \ref{lem:prime quotient}.
\end{proof}

We are now ready to prove the final case of Theorem \ref{thm:main}.

\begin{prop}\label{prop:not finite over center}
    Let $A$ be an affine algebra such that $Z(A)$ is infinite-dimensional and $A$ is not a finitely generated module over $Z(A)$. Then $U(\Inn(A))$ is not noetherian. Consequently, $U(\Der(A))$ is not noetherian.
\end{prop}
\begin{proof}
    For a finitely generated algebra $A$ that is not a finite module over its center, there is always an ideal $I$ that is maximal with respect to the property that $A/I$ is not a finite module over its center. Explicitly, let
    $$X = \{I \ideal A \mid A/I \text{ is not finitely generated as a module over } Z(A/I)\}.$$
    Certainly, $0 \in X$, so $X \neq \emptyset$.
    
    We now apply Zorn's lemma to show that $X$ has a maximal element. Let $C$ be a chain in $X$ and let $I = \bigcup_{J \in C} J$. Assume, for a contradiction, that $A/I$ is a finitely generated module over $Z(A/I)$. As in the proof of Proposition \ref{prop:finite over center}, it follows by the Artin--Tate lemma that $Z(A/I)$ is affine. Since $A$ and $Z(A/I)$ are affine, the condition that $A/I$ is finitely generated over $Z(A/I)$ can be described in finitely many equations in $A/I$. Thus, we can find an ideal $J \in C$ such that $A/J$ is a finitely generated $Z(A/J)$-module, a contradiction. It follows that $I \in X$, so $X$ has a maximal element.

    Let $I$ be a maximal element in $X$, and let $B = A/I$. By construction, $B/J$ is finitely generated over its center for all nonzero ideals $J$ of $B$. Since $\Inn(A) \to \Inn(A/I) = \Inn(B)$ is surjective, it suffices to show that $\Inn(B)$ has a non-noetherian enveloping algebra.
    
    Assume, for a contradiction, that $U(\Inn(B))$ is noetherian. By Proposition \ref{prop:finite-dimensional center}, it must be the case that $Z(B)$ is infinite-dimensional. Furthermore, $B$ cannot be prime, by Lemma \ref{lem:noncommutative prime}. Thus, there exist nonzero ideals $J_1$ and $J_2$ of $B$ such that $J_1 J_2 = 0$. Therefore, $(J_1 \cap J_2)^2 = 0$.

    Suppose $B$ is semiprime. It must be the case that $J_1 \cap J_2 = 0$, so we have an embedding
    \begin{align*}
        B &\inj B/J_1 \times B/J_2 \\
        b &\mapsto (b + J_1, b + J_2).
    \end{align*}
    The rings $B/J_1$ and $B/J_2$ are finitely generated over their centers by construction, so they are PI. Therefore, $B/J_1 \times B/J_2$ is also PI. Finally, we conclude that $B$ is PI since it is a subring of a PI ring.

    Now, $B$ is an affine semiprime PI ring, so $B$ has finitely many minimal primes \cite[Corollary 13.4.4]{McConnellRobson}. By Lemma \ref{lem:noncommutative prime}, every prime quotient of $B$ is either commutative or finite-dimensional, because the natural map
    \begin{align*}
        \Inn(B) &\to \Inn(B/P) \\
        \ad_b &\mapsto \ad_{b + P}
    \end{align*}
    is surjective. We therefore let $P_1,\ldots,P_n$ be the minimal primes of $B$ such that $B/P_i$ is commutative, and let $Q_1,\ldots,Q_m$ be the minimal primes of $B$ such that $B/Q_i$ is not commutative (hence finite-dimensional). Let $P = \bigcap_{i=1}^n P_i$ and $Q = \bigcap_{i=1}^m Q_i$. Since $B$ is semiprime, we have $P \cap Q = 0$. Let
    $$\varphi \colon B \to \prod_{i=1}^m B/Q_i$$
    be the natural map. Certainly, $\dim_\kk(\im(\varphi)) \leq \sum_{i=1}^n \dim_\kk(B/Q_i) < \infty$, so $Q = \ker(\varphi)$ has finite codimension in $B$. Let $x \in Q$ and $b \in B$. Then $[x,b] \in Q$. Furthermore, $B/P_i$ is commutative for all $i \in \{1,\ldots,n\}$, so $[x,b] \in P$. It follows that $[x,b] = 0$, since $P \cap Q = 0$, and thus $x \in Z(B)$. This implies that $Q \subseteq Z(B)$, so $\dim_\kk(B/Z(B)) \leq \dim_\kk(B/Q) < \infty$, a contradiction.
    
    Therefore, $B$ cannot be semiprime. It follows that $B$ has a nonzero nilpotent ideal $N$. Let $\varphi \colon B \surj B/N$ be the canonical surjection, and let $U = \varphi^{-1}(Z(B/N))$. Then $[U,U] \subseteq N$, so $U$ is a solvable Lie algebra, since $N$ is nilpotent. Hence, $\pi(U)$ is finite-dimensional by Proposition \ref{prop:non-noetherian}, where $\pi \colon B \surj \Inn(B) \cong B/Z(B)$ is the natural map. Thus, there is a finite-dimensional space $W \subseteq B$ such that $U = W + Z(B)$. In other words, $Z(B)$ has finite codimension in $U$.
    
    It follows that $\varphi(Z(B))$ has finite codimension in $\varphi(U) = Z(B/N)$. In particular, $Z(B/N)$ is a finite module over $\varphi(Z(B))$. Since $B/N$ is a finite module over $Z(B/N)$, it is also a finite module over $\varphi(Z(B))$. Letting $a_1,\ldots,a_n \in B$ be lifts of a set of generators for $B/N$ as a $\varphi(Z(B))$-module, every element of $B$ is expressible in the form
    $z_1 a_1+\cdots + z_n a_n + y$ with $z_i \in Z(B)$ and $y \in N$. But now $\pi(N)$ is a nilpotent subalgebra of $\Inn(B)$, so $\pi(N)$ is finite-dimensional by Proposition \refnum{prop:non-noetherian}{item:solvable}. Picking lifts $b_1,\ldots,b_m \in N$ of a basis for $\pi(N)$, we see that $y \in Z(B) + \kk b_1 + \cdots + \kk b_m$, so
    $$B = Z(B) + Z(B)a_1 + \cdots + Z(B)a_n + Z(B)b_1 + \cdots + Z(B)b_m,$$
    contradicting the assumption that $B$ is not a finite module over its center. This concludes the proof.
\end{proof}

We are finally ready to fully prove Theorem \ref{thm:main}.

\begin{proof}[Proof of Theorem \ref{thm:main}]
    If $A$ is commutative, then the result follows from Proposition \ref{prop:commutative}. If $A$ is not commutative and finitely generated as a module over its center, then the result follows from Proposition \ref{prop:finite over center} and Lemma \ref{lem:finite over center}. Finally, if $A$ is not finitely generated as a module over its center, then the result follows from Propositions \ref{prop:finite-dimensional center} and \ref{prop:not finite over center}.
\end{proof}

\section{Enveloping algebras of associative algebras}

It follows easily from the proof of Theorem \ref{thm:main} that enveloping algebras of infinite-dimensional associative algebras (viewed as Lie algebras with commutator bracket) are not noetherian.

\begin{cor}\label{cor:associative}
    Let $A$ be an infinite-dimensional $\kk$-algebra, and view $A$ as a Lie algebra with commutator bracket. Then $U(A)$ is not noetherian.
\end{cor}
\begin{proof}
    If $A$ is not affine, then it has an infinite chain of $\kk$-subalgebras. This is also a chain of Lie subalgebras, so $U(A)$ is not noetherian by Proposition \refnum{prop:non-noetherian}{item:ACC on subalgebras}. So, we may assume that $A$ is affine.

    If $\dim_\kk(Z(A)) = \infty$, then $Z(A)$ is an infinite-dimensional abelian Lie subalgebra of $A$, so $U(A)$ is not noetherian by Proposition \refnum{prop:non-noetherian}{item:abelian subalgebra}.

    If $\dim_\kk(Z(A)) < \infty$, then $U(\Inn(A)) \cong U(A/Z(A))$ is not noetherian by Proposition \ref{prop:finite-dimensional center}. Since $U(A)$ surjects onto $U(A/Z(A))$, it follows that $U(A)$ is not noetherian.
\end{proof}

\end{document}